\documentclass{dcds-b}
\usepackage{amsmath}
  \usepackage{paralist}
  \usepackage{graphics} 
  \usepackage{epsfig} 
\usepackage{graphicx}  \usepackage{epstopdf}
\usepackage{chemarr}
 \usepackage[colorlinks=true]{hyperref}
\hypersetup{urlcolor=blue, citecolor=red}

\def\currentvolume{20}
 \def\currentissue{5}
  \def\currentyear{2015}
   \def\currentmonth{July}
    \def\ppages{X--XX}
     \def\DOI{10.3934/dcdsb.2015.20.xx}

\newtheorem{theorem}{Theorem}[section]

\newtheorem{proposition}{Proposition}

\theoremstyle{definition}
\newtheorem{definition}[theorem]{Definition}
\newtheorem{remark}{Remark}

\newcommand{\mint}[0]{\int_\Omega}
\newcommand{\tint}[0]{\int_0^t}
\newcommand{\Tint}[0]{\int_0^T}
\newcommand{\epsi}[0]{\varepsilon}
\newcommand{\eps}[0]{\sigma}
\newcommand{\ff}[0]{\varphi}
\newcommand{\fn}[0]{\ff_n}

\newcommand{\mn}[0]{\mu_n}
\newcommand{\barf}[0]{\bar{\varphi}}
\newcommand{\barm}[0]{\bar{\mu}}
\newcommand{\dert}[0]{\frac{d}{dt}}
\newcommand{\R}[0]{\mathbb{R}}
\newcommand{\rd}[0]{\mathbb{R}^d}
\newcommand{\frif}[0]{\bar{\varphi}^*}
\newcommand{\Ltv}[0]{L^2(0,T;V)}
\newcommand{\Lth}[0]{L^2(0,T;H)}
\newcommand{\tow}[0]{\rightharpoonup}
\newcommand{\EE}[0]{\mathcal{E}(\ff(t))}
\newcommand{\la}[0]{\lambda}
\newcommand{\EEj}[0]{\mathcal{E}(\ff_j(t))}

\allowdisplaybreaks[4]
\def\theequation{\thesection.\arabic{equation}}

\title[Nonlocal Cahn-Hilliard with reaction terms] 
      {Convective nonlocal Cahn-Hilliard equations with reaction terms}

\author[Francesco Della Porta and Maurizio Grasselli]{}


 \email{dellaporta@maths.ox.ac.uk}
 \email{maurizio.grasselli@polimi.it}

\begin{document}
\maketitle

\centerline{\scshape Francesco Della Porta }
\medskip
{\footnotesize
 \centerline{Mathematical Institute, University of Oxford}
   \centerline{Oxford OX2 6GG, UK}
} 

\medskip

\centerline{\scshape Maurizio Grasselli}
\medskip
{\footnotesize
 \centerline{Dipartimento di Matematica, Politecnico di Milano}
   \centerline{Milano 20133, Italy}
}

\bigskip


\begin{abstract}
We introduce and analyze the nonlocal variants of two Cahn-Hilliard type equations with reaction terms. The first one is the so-called Cahn-Hilliard-Oono equation which models, for instance, pattern formation in diblock-copolymers as well as in binary alloys with induced reaction and type-I superconductors. The second one is the Cahn-Hilliard type equation introduced by Bertozzi et al. to describe image inpainting. Here we take a free energy functional which accounts for nonlocal interactions. Our choice is motivated by the work of Giacomin and Lebowitz who showed that the rigorous physical derivation of the Cahn-Hilliard equation leads to consider nonlocal functionals.
The equations also have a transport term with a given velocity field and are subject to a homogenous Neumann boundary condition for the chemical potential, i.e., the first variation of the free energy functional. We first establish the well-posedness of the corresponding initial and boundary value problems in a weak setting. Then we consider such problems as dynamical systems and we show that they have bounded absorbing sets and global attractors.
\end{abstract}

\noindent\\
\section{Introduction}

Consider an incompressible isothermal and isotropic binary solution of constant molar volume which is initially homogeneous. If the temperature is rapidly quenched under a critical value $T_c$, the two solutes, A and B, begin to separate from each other and the domain is split in A-rich and B-rich subdomains. The early stage of this process in which the solution is called spinodal decomposition.
In order to descrive this phenomenon, a  well-known mathematical approach was introduced by J.W. Cahn and J.E. Hilliard (see~\cite{CH 58}) who introduced an appropriate free energy functional. More precisely, given a bounded domain $\Omega\subset \R^d$ with $d\leq 3$ occupied by a binary mixture of components A and B with a mass fraction of $\ff_A(x)$ and $\ff_B(x)$, respectively, setting $\ff:=\ff_A-\ff_B$, the free energy functional
is given by
\begin{equation}
\label{cahnH}
 E(\ff):=\mint \Bigl( \frac{\epsi^2}{2}|\nabla\ff|^2+F(\ff)\Bigr) dx.
\end{equation}

\noindent The term $\frac{\epsi^2}{2}|\nabla\ff|^2$ is a surface tension term
and $\epsi$ is proportional to the thickness of the diffused interface.
$F$ is a double well potential energy density which favors the separation of phases, that is, the formation of A-rich and B-rich regions. Although the physically relevant potential $F$ is logarithmic (see \cite[(3.1)]{CH 58}), in the literature it is usually approximated, for instance, by the following polynomial function
\begin{equation}
\label{dwp poli}
F(s)=\frac{1}{4}(s^2-1)^2.
\end{equation}
\noindent
Once an energy is defined, the spinodal decomposition can be viewed as an energy minimizing process.
Thus, the evolution of the phenomenon can be described as a gradient flow (cf.~\cite{C 61}, see also~\cite{Fife,CHRassegna3})
\[
\ff_t - \nabla\cdot \bigl( M(\ff)\nabla \mu_L \bigr)=0 ,\qquad\text{in } \Omega\times (0,T),
\]
on a given time interval $(0,T)$, $T>0$, where $M(\ff)$ is called mobility and the chemical potential $\mu_L$
is the first variation of $E$, that is,
\[
\mu_L:=\frac{\delta E}{\delta \ff}=F'(\ff)-\epsi^2\Delta\ff.
\]
\noindent
This is the well-known Cahn-Hilliard equation which has been widely studied by many authors. Here we just mention some pioneering works \cite{CHAnalysis3, CHAnalysis1, CHAnalysis2, Z 86} for the case with constant mobility, while for nonconstant (and degenerate at pure phases) mobility the basic reference
is~\cite{EG 96} (cf. also~\cite{Schi07} for nondegenerate mobility). For further references see, for instance, \cite{CHRassegna1,CHRassegna2,CHRassegna3}.

However, the purely phenomenological derivation of the Cahn-Hilliard equation is somewhat unsatisfactory from the physical viewpoint. This led Giacomin and Lebowitz to study the problem of phase separation from the microscopic viewpoint using a statistical mechanics approach (see~\cite{GL 3} and also~\cite{GL 1}). Performing the hydrodynamic limit they deduced a continuum model which is a nonlocal version of the Cahn-Hilliard equation, namely, they found the following
\[
\ff_t=\nabla\cdot\bigl(M(\ff)  \nabla\mu
\bigr)\bigr),\qquad\text{in } \Omega\times (0,T),
\]
where $\mu$ is defined by
\begin{equation}
\label{chempot}
\mu:=\frac{\delta \mathcal{E}}{\delta\ff}=a(x)\ff+F'(\ff)-J\ast\ff,\qquad\text{in } \Omega\times (0,T),
\end{equation}
where $J: \mathbb{R}^d \to \mathbb{R}$ is a convolution kernel such that $J(x)=J(-x)$ and $a(x):=\mint J(x-y)dy$.
In this case, the free energy functional $\mathcal{E}$ is given by
\[
\mathcal{E}(\ff)=\frac{1}{4}\int_\Omega\mint J(x-y)(\ff(x)-\ff(y))^2dxdy  + \eta\mint F(\ff(x))dx
\]

It is worth observing that $E(\ff)$ can be recovered from $\mathcal{E}(\ff)$ using Taylor expansion for the term $(\ff(x)-\ff(y))$ and supposing that $J$ is sufficiently peaked around zero. Thus the original Cahn-Hilliard equation can be seen as an approximation of its local version. Furthermore, as discussed in~\cite{GL 2}, the corresponding sharp interface model is the same as the local Cahn-Hilliard equation.

The nonlocal Cahn-Hilliard equation with degenerate mobility and logarithmic potential was first studied rigorously in \cite{GL 2}. Then more general results were proven in \cite{NLCHAnalysis1}. The longtime behavior of single solutions was studied in \cite{NLCHAA1, NLCHAA2}, while the global longterm dynamics (say, existence of global and exponential attractors) was investigated in \cite{NLCHExponential}. The case of constant mobility and smooth potential was treated in~\cite{CH forte, CH forte D}.

More recently, some works have been devoted to study the nonlocal Cahn-Hilliard equation coupled with the Navier-Stokes system (see~\cite{grass, grass trattore, FG2}). This is a nonlocal variant of a well-known diffuse interface model for phase separation in two-phase incompressible and isothermal fluids (model H). All such contributions produced, as by-products, further results on the nonlocal Cahn-Hilliard equation.

Taking advantage of some of these results, here we want to introduce and study two nonlocal Cahn-Hilliard equations with reaction terms which have already been proposed and analyzed in the local case: the Cahn-Hilliard-Oono (CHO) equation and the Cahn-Hilliard-Bertozzi-Esedo\={g}lu-Gillette (CHBEG) equation (see also \cite{CHRBio,miranville2} for other types of reaction terms).

Let us introduce the former first.

In~\cite{Glotzer} the authors introduce a reaction between the molecules or the polymers in the mixture in order to stabilize and control the length of the patterns during spinodal decomposition.
Let us suppose that the mixture consists of molecules or polymers of type A and B and that a chemical reaction $A\xrightleftharpoons[\Gamma_2]{\Gamma_1} B$ is induced during the spinodal decomposition. Here $\Gamma_1$ and $\Gamma_2$ are, respectively, the temperature dependent forward and backward reaction rates. By defining $\ff$ as the difference between the concentrations, the following Cahn-Hilliard type equation is proposed (see~\cite{Glotzer})
\begin{equation}
\label{LCHO}
\ff_{t}=\Delta\mu_L - (\Gamma_1+\Gamma_2)\Bigl(\ff-\frac{\Gamma_1-\Gamma_2}{\Gamma_1+\Gamma_2}\Bigr).
\end{equation}
This equation is known as the CHO equation and was introduced in~\cite{BO 90} (cf. \cite{OP 87}, however) for phase separation in diblock copolymers (see also \cite{OK 86}). It can be viewed as a gradient flow generated by the energy functional:
\[
\mathcal{H}(\ff)=\frac{\epsi^2}{2}\mint|\nabla\ff|+\mint F(\ff) + \frac{\sigma}{2}\mint\mint G(x,y) (\ff(x)-m)(\ff(y)-m),
\]
where $m\in\mathbb{R}$, $\sigma$ and $\epsi$ are two positive constants and $G$ is the Green function of $-\Delta$ with either periodic or no-flux boundary conditions, i.e.,
\[
-\Delta G(x,y)=\delta(x-y) ,\qquad\text{in } \Omega.
 \]
This energy differs from the one in \eqref{cahnH} just in the last term, which is nonlocal, and prefers rapid oscillations between the two polymers (see~\cite{VG 10} for the one-dimensional case).

It has been observed that far from the domain border, pattern are not influenced by the boundary conditions. The most commonly used ones are either periodic or no-flux.
The latter can be mathematically expressed as follows:
\[
\frac{\partial \ff}{\partial \textbf{n}}=0 \quad \text{and}\quad \frac{\partial\mu_L}{\partial \textbf{n}}=0, \qquad\text{on } \partial\Omega\times (0,T).
\]
If we integrate equation \eqref{LCHO} on the domain $\Omega$
with such boundary condition we have
\begin{equation}
\label{mass perde intro}
\dert\barf = -\sigma(\barf-m),
\end{equation}
where $\barf:=\frac{1}{|\Omega|}\mint\ff.$ This equation is a linear ODE with solution $\barf(t)=\barf(0)e^{-\sigma t}+ m$.
Note that, when $\sigma>0$, then $\barf(0)=m$ implies the conservation of the total mass is conserved. This happens
in phase separation of diblock copolymers. Instead, if $\barf(0)\neq m$ then $\barf(t)$ converges to $m$ exponentially fast.

More information and results about the CHO equation and its application can be found in \cite{BO 90, choksi peletier, choksi 01, CMW11, CR 03, Glotzer, Lodge, nanolit, muratov 10, Muratov 02, NO 95, OK 86, antireflex} and references therein
(see also ~\cite{reazione cinese, reazione cinese2} for hydrodynamic effects). It is worth noting that the pattern formation process differs from the one associated with the classical Cahn-Hilliard equation (see \cite{AKW13}). As far as well-posedness and global longtime behavior (i.e., existence of global and exponential attractors) are concerned, we refer to~\cite{miranville} (cf. also \cite{BGM13}).

The CHBEG equation can be viewed as a generalized version of the CHO equation. It has been introduced by Bertozzi et al. to model image inpainting, i.e., the restoration process of damaged images and \textit{super-resolution} (see~\cite{bertozzi2,bertozzi1}) with regard to binary images or images with high contrast. As in the case of phase separation, the double well potential tends to separate the mixture in pure phases, so it does with binary images by separating the two colors of the image without many shades.

The CHBEG equation reads as follows
\[
\ff_t=\Delta(-\epsi^2\Delta\ff+F'(\ff))-\lambda(x)(\ff - h(x)),\qquad\text{in } \Omega\times (0,T).
\]
Here $\Omega$ is a two-dimensional domain representing the whole image, with no-flux boundary conditions. The function $F$ is still a double well potential, with two minima $\ff_1$ and $\ff_2$, that is, the two colors of the image. The function $h$ represents the image that needs to be restored ($h(x)=\ff_i(x)$ where the color is $\ff_i$, $i=1,2$. The function $\lambda$ is a penalization parameter which forces the restored image $\ff$ to be equal to $h$ outside the damaged region. The numerical simulations carried out in~\cite{bertozzi2} show that the implemented method is faster then its main competitors and it gives satisfactory results.

Existence and uniqueness of a weak solution to the Cahn-Hilliard equation for the CHBEG equation have been proven in~\cite{bertozzi1}, while the longtime behavior of its solutions has recently been analyzed in~\cite{miranville pre}.

On account of what we have observed concerning the derivation of the nonlocal Cahn-Hilliard equation, it seems natural to study the nonlocal variants of CHO and CHBEG equations.

More precisely, we consider first a nonlocal CHO equation with a transport term which accounts for a possible flow of the mixture at a certain given (averaged) velocity field $u:\Omega\times(0,T)\to\rd$, that is
\begin{equation}
\label{CHO}
\ff_t + \sigma( \ff - \frif) + \nabla\cdot(u \ff)= \Delta \mu + g,\qquad\text{in } \Omega\times (0,T).
\end{equation}
where $\mu$ is given by  \eqref{chempot} and  $g$ is a given external source.
Here we have set all the constants equal to one but $m$ which has now been denoted by $\frif$
on account of  \eqref{mass perde intro}.

Then, in the same spirit, we introduce the nonlocal version of the CHBEG equation, that is,
\begin{equation}
\label{CHBEG}
\ff_t +\lambda(x)(\ff-h(x) )+\nabla\cdot(u\ff)=\Delta\mu + g, \qquad\text{in } \Omega\times(0,T).
\end{equation}
As for the CHO equation, we include a transport term which can be used to improve the inpaint of the image and an external source.
Therefore we consider the following initial and boundary value problems
\begin{equation}
\label{problema}
\begin{cases}
\displaystyle\ff_t+  \nabla \cdot (u \ff) +  \eps (\ff-\frif) =
\Delta \mu +g &\text{in } \Omega\times(0,T)\\[1.4ex]
\displaystyle\mu=a\ff-J\ast\ff+ F'(\ff) &\text{in } \Omega\times(0,T)\\[1.4ex]
\displaystyle\frac{\partial\mu}{\partial n}=0 &\text{on } \partial\Omega\times (0,T) \\[1.4ex]
\displaystyle \ff(0)=\ff_0 &\text{in } \Omega
\end{cases}
\end{equation}
and
\begin{equation}
\label{CHB}
\begin{cases}
\displaystyle\ff_t +\lambda(x)(\ff-h(x) )+\nabla\cdot(u\ff)=\Delta\mu + g
&\text{in } \Omega\times(0,T)\\[1.4ex]
\displaystyle\mu=a\ff-J\ast\ff+ F'(\ff) &\text{in } \Omega\times(0,T)\\[1.4ex]
\displaystyle\frac{\partial\mu}{\partial n}=0 &\text{on }\partial\Omega\times (0,T) \\[1.4ex]
\displaystyle \ff(0)=\ff_0 &\text{in } \Omega.
\end{cases}
\end{equation}

The first goal will be to define a notion of weak solution to problems \eqref{problema} and \eqref{CHB} and to establish their well-posedness under suitable conditions. This will be done in Sections \ref{EUCHO} and \ref{EUCHBEG}. Then, in Sections \ref{GACHO} and \ref{GACHBEG}, we will analyze nonlocal CHO and CHBEG equations equipped with no-flux boundary condition as dynamical systems on suitable phase spaces. In particular, we will prove the existence of a (bounded) absorbing set as well as of the global attractor. The former result will require some care (especially for the CHBEG equation) to handle the bound on the spatial average of $\ff$ which is not conserved in general.

We also expect to prove further results. for instance, the existence of more regular solutions, the regularization of weak solutions and the existence of exponential attractors as shown in \cite{NLCHExponential} (see also \cite{FGG}). These issues will possibly be the subject of a forthcoming paper. A more challenging question appears to be the following: does a solution converge to a single equilibrium? This happens for the nonlocal Cahn-Hilliard equation (see, e.g., \cite{NLCHExponential} and references therein), but it is open even for equation \eqref{LCHO}. We finally mention that further results on well-posedness and regularity in the case of degenerate mobility and logarithmic potential have recently been proven in \cite{melrocca}.
That work has been motivated by the present contribution.

\section{Nonlocal CHO equation: Well-posedness of  \eqref{problema}}
\label{EUCHO}
\setcounter{equation}{0}
\noindent
Here we show that there exists a unique weak solution to problem \eqref{problema} under suitable assumptions on $u$, $g$, $J$ and $F$. We also prove that any weak solution satisfies an energy equality and it depends continuously with respect to the initial datum in a suitable norm. We adapt the strategy used in~\cite{grass}.

\subsection{Notation}
The space $L^2(\Omega)$ will be denoted $H$ and the space
$H^1(\Omega)$ will be denoted $V$. Furthermore we will set:
\[
V_2=\Bigl\{v \in H^2(\Omega):\, \frac{\partial v}{\partial
n}=0\qquad on\, \partial\Omega\Bigr\}.
\]
By $\Arrowvert\,\cdot\,\Arrowvert$ and $(\cdot\,,\cdot)$ will be denoted respectively the norm and the scalar product in $L^2(\Omega)$.

The linear operator $A=-\Delta: V_2 \to H$ is self-adjoint and non-negative. Moreover, it is strictly positive if restricted to $H_0=\{\psi\in H\,:\, \bar{\psi} =0\}$.
Then we will also set $\Arrowvert\, A^{-1/2}( \cdot)\,\Arrowvert $  and we observe that the norm $\|\cdot\|_\#$ defined as
\[
\|x\|_\#:=\Bigl(\|x-\bar{x}\|^2_{-1}+\bar{x}^2 \Bigr)^\frac{1}{2},
\]
is equivalent to the usual norm of $V'$.

Furthermore, $c$ will indicate a generic nonnegative constant depending on $J,\, F,\, \Omega$, $\sigma,\,\frif$, $g,\,\lambda,\, h$ and $u$ at most. If it is necessary to distinguish between such constants, then we will use $C_i$. Occasionally we will point out if such constants also depend on other quantities but time. If there is a further dependence on $T$ and/or the initial datum $\ff_0$, then we shall use $N$. The value of $c$ and $N$ may vary even within the same line.

\subsection{Assumptions and statements}
Let us introduce first the assumptions following~\cite{CH forte} and~\cite{grass} (cf. also~\cite{miranville}):

\begin{description}
	\item[(H0)]
		 $\Omega\subset\mathbb{R}^d$, $d\leq 3$, open, bounded and connected with a smooth boundary.
	\item[(H1)]
		$J\in W^{1,1}(\rd)$ is such that:
		\[
         J(x)=J(-x), \qquad a(x):=\mint J(x-y)\,dy \geq 0, \;\text{a.e.}\, x\in\Omega.
		\]
	\item[(H2)]
		$F\in C^{2,1}_{loc}(\R)$ and there exists $c_0 > 0$ such that
		\[
F''(s)+a(x)\geq c_0, \quad \forall s\in \R,\; \text{a.e.}\, x \in\Omega.
\]
	\item[(H3)]
There exist $ c_1>\frac{1}{2}\|J\|_{L^1(\rd)}$ and $c_2\in \R$
such that
		\[
		F(s)\geq c_1 s^2-c_2, \quad \forall s \in \R.
		\]
	\item[(H4)]
		There exist $c_3>0$ and $p\in (1,2]$ such that
		\[
		|F'(s)|^p\leq c_4(|F(s)|+1),\quad \forall s\in \R.
		\]
	\item[(H5)]
		$u\in L^2(0,T; (L^\infty(\Omega)\cap H^1_0(\Omega))^d)$.	
	\item[(H6)]
		$\eps$ is a given positive constant.
	\item[(H7)]
		$\frif\in\R$ is a given constant.
	\item[(H8)]
		$g\in L^2(0, T; V')$.
\end{description}


\begin{remark}
\label{perturbation}
Assumption (H2) implies that the potential $F$ is a quadratic
perturbation of a strictly convex function. Indeed $F$ can be
represented as
\begin{equation}
\label{F coerciva}
F(s)=G(s)-\frac{a^*} {2}\,s^2
\end{equation}
with $G\in C^{2,1}(\R)$ strictly convex, since $G''\geq c_0$ in
$\Omega$. Here $a^*= \| a\| _{L^{\infty}(\Omega)}$ and observe that $a\in L^\infty(\Omega)$ derives from (H1).
\end{remark}	

\begin{remark}
\label{growth}
Since $F$ is bounded from below, it is easy to see that (H4) implies that $F$ has polynomial growth of order $p'$, where $p'\in[2,\infty)$ is the conjugate index to p. Namely there exist $c_4>0$ and $c_5\geq 0$ such that
\[
|F(s)|\leq c_4|s|^{p'}+c_5,\qquad \forall s\in\R.
\]
\end{remark}	

\begin{remark}
The potential  \eqref{dwp poli} satisfies all the hypotheses on $F$.
\end{remark}	

\begin{remark}
\label{h4}
It is easy to show that (H4) implies
		\[
		|F'(s)|\leq c(|F(s)|+1),\quad \forall s\in \R;
		\]
furthermore (H3) implies that
		\[
		|F(s)|\leq F(s)+ 2\max\{0,c_2\},\quad \forall s\in \R.\\\\
		\]
\end{remark}

\begin{remark}
Take for simplicity $u=g=0$ in equation  \eqref{CHO} and observe that it can formally be rewritten as follows
\begin{equation*}
\ff_t = \nabla\cdot\bigl( (F''(\ff)+a)\nabla\ff\bigr)+\nabla\cdot\bigl(\nabla a\ff\bigr)-\nabla J\ast \ff +\eps(m-\ff)
\end{equation*}
from which the crucial role of (H2) is evident. This remark also holds for  \eqref{CHBEG}.
\end{remark}

We can now give our definition of weak solution to problem \eqref{problema}:
\begin{definition}
\label{soluzione debole}
Let $\ff_0\in H$ such that $F(\ff_0)\in L^1(\Omega)$ and $T>0$ be given.
Then $\ff$ is a weak solution to problem \eqref{problema} on
$[0,T]$ corresponding to $\ff_0$ if:
	\begin{align*}
	&\ff\in L^\infty(0,T;H)\cap L^2(0,T;V),\;
	\ff_t\in L^2(0,T;V'),\\
	&\mu=a\ff-J\ast\ff+ F'(\ff) \in \Ltv;
	\end{align*}
and, setting
	\begin{equation}
	\label{eq rho}
	\rho(x,\ff):=a(x)\ff+F'(\ff),
	\end{equation}
we have
	\begin{align}
	\label{eq}
&\langle\ff_t,\psi\rangle+(\nabla\rho,\nabla\psi)+(\eps(\ff-\frif),\psi)=(\nabla
J\ast\ff,\nabla \psi)+(u\ff,\nabla\psi)+\langle g, \psi\rangle;\\
	&\label{ci}
	\ff(0)=\ff_0,
	\end{align}
for every $\psi\in V$ and for almost any $t\in(0,T)$.
\end{definition}

\begin{remark}
Since $\rho=\mu+J\ast\ff$, we have $\rho\in L^2(0,T;V)$. Besides
$\ff\in C([0,T]; H)$.
\end{remark}

\begin{remark}
Observe that if we choose $\psi=1$ in \eqref{eq} and we set $f:=\bar{g} +\eps\frif$, then we obtain
\[
\dert \barf + \eps\barf = f.
\]
Thus the total mass is not conserved in general.
\end{remark}

Existence and uniqueness for problem \eqref{problema} is given by:
\begin{theorem}
\label{buona posizione}
Let $\ff_0\in H$ such that $F(\ff_0)\in L^1(\Omega)$ and suppose(H0)-(H8) are satisfied. Then, for every $ T>0$ there exists a
unique weak solution $\ff$ to problem  \eqref{problema} on
$[0,T]$ corresponding to $\ff_0$.
Furthermore, $F(\ff)\in L^\infty(0,T;L^1(\Omega))$ and setting
\begin{equation}
\label{nlEnergy}
\mathcal{E}(\ff(t))=\frac{1}{4}\mint\mint J
(x-y)(\ff(x,t)-\ff(y,t))^2 \,dx\,dy + \mint F(\ff(x,t))\,dx.
\end{equation}
the following energy equality holds for almost every $t\in(0,T)$:
\begin{equation}
\label{energy equality}
\dert \mathcal{E}(\ff(t)) + \|\nabla\mu\|^2 + \eps(\ff-\frif,\mu) = (u\ff,\nabla\mu)+\langle g ,\mu\rangle.
\end{equation}
\end{theorem}
\noindent
Furthermore, we have:
\begin{proposition}
\label{stability}
Let the hypotheses (H0)-(H8) hold. Then the weak solution to \eqref{problema} continuously depend on $\ff_0$
with respect to the $\|\cdot\|_\#$-norm.
\end{proposition}

\subsection{Proof of Theorem~\ref{buona posizione} and of Proposition~\ref{stability}}
This section is divided into three parts. The first part is devoted to the existence of a weak solution. The energy identity \eqref{energy equality} is shown in the second one and the continuous dependence estimate is deduced in the third.

\subsubsection{Existence of a solution}
Existence of a weak solution is carried out following closely~\cite{grass}. We will first prove the existence of the solution when $\ff_0\in V_2$.
Then we will recover the existence of a solution when $\ff_0\in H$ such that $F(\ff_0)\in L^1(\Omega)$ by a density argument. We also assume
$u\in C([0,T]; (L^\infty(\Omega)\cap H^1_0(\Omega))^d)$ and $g\in C([0,T]; V')$ for the moment.

As usual in the Faedo-Galerkin method, we consider the family $\{\psi_j\}_{j}\subset V_2$ of the
eigenvectors of the self-adjoint, positive, linear operator $A+I: V_2 \to H$.
Let us define the n-dimensional subspace
$\Psi_n:=\langle\psi_1,...,\psi_n\rangle$ and consider the
orthogonal projector on this subspace $P_n:=P_{\Psi_n}$.
We then look for two functions of the following form:
\[
\ff_n(t)=\sum_{k=1}^{n} b_k^{(n)}(t)\psi_k, \qquad
\mu_n(t)=\sum_{k=1}^{n} c_k^{(n)}(t)\psi_k
\]

\noindent
that solve the following discretized problem
\begin{align}
\label{eq approssimata}
&(\ff_n',\psi)+(\nabla\rho_n,\nabla\psi)+(\eps(\ff_n-\frif),\psi)
=(u\ff_n,\nabla\psi)+(\nabla
J\ast\ff_n,\nabla\psi) + \langle g, \psi\rangle\\
\label{defrho}
&\rho_n:=a(\cdot)\ff_n+F'(\ff_n),\\
\label{apprmu}
&\mu_n=P_n(\rho_n-J\ast\ff_n),\\
\label{ci approssimata}
&\ff_n(0)=\ff_{0n},
\end{align}
\noindent
for every $\psi\in\Psi_n$, and where $\ff_{0n}:=P_n\ff_0$.

By using the definition of $\ff_n$, problem  \eqref{eq approssimata}-\eqref{ci approssimata} becomes equivalent to a Cauchy problem for a system of ordinary
differential equations in the $n$ unknowns $b_i^{(n)}$.
Thanks to (H2), the Cauchy-Lipschitz theorem yields that there exists a unique solution $b^{(n)}\in C^1([0,T^*_n];\R^n)$
for some maximal time $T^*_n\in(0,+\infty]$.

Let us show that $T^*_n=+\infty$, for all $n\geq 1$. Indeed, taking $\psi=\mu_n$ as test function in \eqref{eq approssimata}
we get the following identity:
\[
(\ff_n',\mu_n)+(\nabla\rho_n,\nabla\mu_n)+(\eps(\ff_n -\frif ),\mu_n)
=(u\ff_n,\nabla\mu_n)+(\nabla
J\ast\ff_n,\nabla\mu_n) + \langle g,\mn\rangle,
\]
which can be rewritten as follows
\begin{align}
\nonumber
&\frac{d}{dt}\Bigl( \frac{1}{4}\mint\mint
J(x-y)(\ff_n(x)-\ff_n(y))^2+\mint F(\ff_n) \Bigr)+
\|\nabla\mu_n\|^2
+(\eps(\ff_n-\frif),\mu_n)\\
\label{stima approssimata}
&=(\nabla J\ast\ff_n,\nabla\mu_n)-(\nabla(P_n(J\ast\ff_n)),\nabla\mu_n)+(u\ff_n,\nabla\mu_n) + \langle g, \mu_n\rangle.
\end{align}

\noindent
It is easy to see that
\begin{equation}
\label{stima2}
(\nabla \mu_n,\nabla P_n(J\ast\ff_n)) \leq
\frac{1}{4}\|\nabla\mu_n\|^2+\|\ff_n\|^2\|J\|_{W^{1,1}}^2
\end{equation}
and
\[
(\nabla J\ast\ff_n,\nabla\mu_n)\leq
\frac{1}{4}\|\nabla\mu_n\|^2+\|\ff_n\|^2\|J\|_{W^{1,1}}^2.
\]

\noindent
By means of (H3), we can deduce the existence of a positive constant $\alpha$ such that $\alpha<2c_1-\|J\|_{L^1(\Omega)}$.
Therefore we have
\begin{align}
\nonumber
&\frac{1}{2}\mint\mint J(x-y) (\ff_n(x)-\ff_n(y))^2\,dx\,dy
+2\mint F(\ff_n)
\\
\label{stima5}
&=\|a\ff_n\|^2 + 2\mint F(\ff_n) - (\ff_n, J\ast \ff_n)
\geq \alpha\Bigl( \| \ff_n\|^2 + \mint F(\fn)\Bigr) - c.
\end{align}

\noindent
Thanks to Remark~\ref{h4} and to the identity
$$
(P_n(-J\ast\fn+a\fn),1)=(-J\ast\fn+a\fn,1)=0
$$
we get
\begin{equation}
\label{mn minore di F}
\biggl | \mint\mn \biggr|
= \bigl|(F'(\fn),1)\bigr|
\leq \mint \bigl|F'(\fn)\bigr|
\leq c\mint F(\fn)+c.
\end{equation}
\noindent
Thus we obtain
\begin{equation}
\label{stima forzante}
(\eps \frif,\mn)\leq \eps\frif (1,\mn) \leq c\eps |\frif| \Bigl( \mint F(\fn)+1\Bigr).
\end{equation}

\noindent
By using (H5), it is easy to show that
\begin{equation}
\label{stima u}
(u\ff_n,\nabla\mu_n)\leq
\|u\|_{L^\infty}\|\nabla\mu_n\|\|\ff_n\|\leq
\frac{1}{8}\|\nabla\mu_n\|^2+2\|u\|_{L^\infty}^2\|\ff_n\|^2.
\end{equation}
On account of (H8) and using the Poincar\'e-Wirtinger inequality, we have
\begin{equation}
\langle g, \mn - \bar{\mu}_n\rangle \leq \|g\|_{V'}\|\mn-\bar{\mu}_n\|_V\leq (1+C_p) \|\nabla\mn\|\|g\|_{V'}
\label{forzante mumedia}
\leq c\|g\|_{V'}^2+\frac{1}{8}\|\nabla\mn\|^2.
\end{equation}
Besides, from \eqref{mn minore di F} we get that
\begin{equation}
\label{forzante mumedia2}
\langle g, \bar{\mu}_n\rangle\leq  c\Bigl(\mint F(\fn)+1\Bigr)\|g\|_{V'}.
\end{equation}
Collecting \eqref{forzante mumedia} and \eqref{forzante mumedia2} we deduce
\begin{equation}
\label{forzante h}
\langle g, \mn\rangle
\leq c\|g\|_{V'}^2+\frac{1}{8}\|\nabla\mn\|^2 +  c\|g\|_{V'}\mint F(\fn)+c.
\end{equation}

On the other hand, thanks to (H6), we have
\[
(\eps\ff_n,\mu_n)
=(\eps\ff_n, a\ff_n+F'(\ff_n)- J \ast
\ff_n)
\geq (\eps\ff_n,F'(\ff_n))- \eps\|J\|_{L^1{(\Omega})} \|\ff_n\|^2.
\]

\noindent
We now exploit the convexity of the function $F(s)+a(x)s^2$ (see Remark~\ref{perturbation})
\[
(\eps\ff_n,F'(\ff_n)+a^*\fn) \geq \eps\mint \bigl( F(\ff_n)+a^*\fn^2\bigr)-\eps|\Omega| F(0)
\]
so we can write
\begin{equation}
\label{stima4}
(\eps\ff_n,\mu_n)
\geq \eps\mint F(\ff_n) - \eps\|J\|_{L^1(\Omega)}
\| \ff_n\|^2-c.
\end{equation}
If we integrate \eqref{stima approssimata} with respect to
time between $0$ and $t\in(0,T_n^*)$, taking  \eqref{stima2}-\eqref{stima5},  \eqref{stima forzante},  \eqref{stima u}, \eqref{forzante h},  \eqref{stima4} into account, we find
\begin{align}
\label{maxinter}
&\alpha \Bigl( \|\ff_n(t)\|^2+\mint F(\fn) \Bigr)+\tint \|\nabla\mu_n(\tau)\|^2\, d\tau \nonumber\\
&\leq M + \tint K(\tau) \Bigl(  \|\ff_n(\tau)\|^2 + \mint F(\fn)\Bigr) \, d\tau,
\end{align}
\noindent
which holds for all $t\in [0,T_n^*),$ where
\[
M=c\Bigl(1 + \|\ff_{0}\|^2 + \mint F(\ff_0)+ \|g \| ^2_{L^2 (0,T;V') }\Bigr),
\]
\[
K=\max{\bigl(2 \|J\|_{W^{1,1}} +
\eps\|J\|_{L^1(\Omega)} + 2\|u\|_{L^\infty
(\Omega)}^2, c(1+\|g\|_{V'})\bigr) \in L^2(0,T)}.
\]
Here, we have used the fact that that $\ff_0$ and $\ff_{0,n}$ are supposed to
belong to $V_2$. We point out that $M$ and $K$ do not depend on $n$.

Thus, inequality  \eqref{maxinter} entails that $T_n^*=+\infty$, for all $n\geq 1$.
As a consequence, \eqref{eq approssimata}-\eqref{ci approssimata}
has a unique global-in-time solution. Furthermore, we obtain the
following estimates, holding for any given $0<T<+\infty$:
\begin{equation}
\label{fn limitata}
\|\ff_n\|_{L^\infty(0,T;H)}\leq N
\end{equation}
\begin{equation}
\label{mn limitata}
\|\nabla\mu_n\|_{L^2(0,T;H)}\leq N
\end{equation}
\begin{equation}
\label{Ffn limitata}
\|F(\ff_n)\|_{L^\infty(0,T;L^1(\Omega))}\leq N
\end{equation}
where $N$ is independent of $n$.

Recalling  \eqref{apprmu}, we get
\begin{equation}
\label{gradfn in V 3}
\frac{c_0}{4}\|\nabla\fn\|^2+\frac{1}{c_0}\|\nabla\mn\|^2\geq
(\nabla\mn,\nabla\ff_n)
\geq \frac{c_0}{2}\|\nabla
\fn\|^2-c\|\fn\|^2,\end{equation}
and \eqref{fn limitata}, \eqref{mn limitata}, \eqref{gradfn in V 3} yield
\begin{equation}
\label{fn in V}
\|\fn\|_{L^2(0,T;V)}\leq N.
\end{equation}

\noindent
The next step is to deduce a (uniform) bound for $\mn$ in $\Ltv$.
Let us observe that, thanks to \eqref{mn minore di F} and \eqref{Ffn limitata}, we have
\begin{equation}
\label{mu in L1}
\mint\mn
\leq \biggl | \mint\mn \biggr|
\leq c\mint \bigl|F(\fn)\bigr|+c
\leq N.
\end{equation}
The Poincar\'e inequality implies
\begin{equation}
\label{mu in L1 2}
\biggl\|\mn-\frac{1}{|\Omega|}\mint\mn \biggl\|  \leq  c\|\nabla\mn\|,
\end{equation}
and from \eqref{mn limitata} and \eqref{mu in L1} we deduce that
\begin{equation}
\label{mn in V}
\|\mn\|_{\Ltv}\leq N.
\end{equation}

\noindent
Observe now that
\[
\|\tilde\rho_n\|^2_V=\|\mn+P_n(J\ast\fn)\|_V^2
\leq 2\|\mn\|_V^2+ 2(\|J\|^2_{L^{1}}+\|\nabla
J\|^2_{L^{1}})\|\fn\|^2,
\]
so that from \eqref{mn in V} we immediately get
\begin{equation}
\label{rho in V}
\|\tilde\rho_n\|_{\Ltv}\leq N.
\end{equation}
Furthermore, recalling  \eqref{defrho} and invoking (H4), we obtain
\[
\|\rho_n\|_{L^p(\Omega)}
\leq c a^*\|\fn\|+\|F'(\fn)\|_{L^p(\Omega)}
\leq  cN+c\Bigl( \mint |F(\fn)| \Bigr)^{1/p}\leq N,
\]
which yields the bound
\begin{equation}
\label{rho in Lp}
\|\rho_n\|_{L^\infty(0,T; L^p(\Omega))}\leq N.
\end{equation}
We finally provide an estimate for the sequence $\fn'$. We take a generic test function $\psi\in V$ and we
write it as $\psi=\psi_1+\psi_2$, where $\psi_1=P_n\psi\in\Psi_n$ and $\psi_2=\psi-\psi_1\in \Psi_n^\perp$.
It is easy to see that
\begin{equation}
\label{stima termine rho}
|(\nabla\rho_n,\nabla\psi_1)|
\leq \|\nabla\rho_n\| \|\nabla\psi_1\|
\leq \|\nabla\rho_n\| \|\nabla\psi\|_V,\\
\end{equation}
\[
|(u\fn,\nabla\psi_1)|
\leq \|u\|_{L^\infty(\Omega)} \|\nabla\psi_1\| \|\fn\|
\leq  N \|u\|_{L^\infty} \|\psi\|_V.
\]

\noindent
The reaction term can be treated as follows
\[
\Bigl| \mint\eps(\fn-\frif)\psi_1 \Bigr| \leq \bigl(|\Omega|\eps\frif+ \eps \|\fn\| \bigr)\|\psi_1\| \leq N \|\psi\|_V,
\]
and for the source term we have
\[
\Bigl| \langle g, \psi_1 \rangle\Bigr| \leq \|g\|_{V'} \|\psi_1\|_V.
\]
By using Young's lemma we infer
\begin{equation}
\label{stima termine J}
\Bigl| \mint\nabla J\ast\fn\nabla\psi_1 \Bigr|
\leq \|\psi\|_V \|\nabla J\|_{L^1(\Omega)}\|\fn\|
\leq  N \|\nabla J\|_{L^1(\Omega)}\|\psi\|_V.
\end{equation}
\noindent
From \eqref{eq approssimata}, owing to \eqref{stima termine rho}-\eqref{stima termine J}, we have that
\[
|(\fn',\psi)|
\leq (N+\|\nabla\rho_n\|+\|g\|_{V'})\|\psi\|_V,
\]
which gives
\begin{equation}
\label{fn' in V'}
\|\fn'\|_{L^2(0,T;V')}\leq N.
\end{equation}

\noindent
Collecting estimates  \eqref{fn limitata},  \eqref{fn in V},  \eqref{mn in V},  \eqref{rho in V},  \eqref{rho in Lp}, \eqref{fn' in V'}, we find
\begin{align}
\label{esistenza 1}
&\ff\in L^\infty(0,T;H)\cap\Ltv\cap H^1(0,T;V'),\\
\label{esistenza 2}
&\mu\in\Ltv,\\
\nonumber
&\tilde{\rho}\in\Ltv,\\
\nonumber
&\rho\in L^\infty (0,T; L^p(\Omega)),
\end{align}
such that, up to a subsequence,
\begin{align}
\label{conv 0}
&\fn\tow \ff \quad \text{weakly*  in } L^\infty(0,T;H),\\
\label{fn conv}
&\fn\tow \ff \quad \text{weakly  in } \Ltv,\\
\label{strong conv}
&\fn\to \ff \quad \text{strongly  in }  L^2(0,T;H) \;\text{and a.e.  in } \Omega\times(0,T)\\
\label{mn conv}
&\mn\tow \mu \quad \text{weakly  in } \Ltv,\\
\label{rhon conv}
&\rho_n\tow \tilde\rho \quad \text{weakly  in }  \Ltv,\\
\label{rho conv}
&\rho_n \tow \rho \quad \text{weakly*  in } L^\infty (0,T; L^p(\Omega)),\\
\label{fn' conv}
&\fn'\tow \ff_t \quad \text{weakly  in }  L^2(0, T; V').
\end{align}
\noindent
We can now pass to the limit in  \eqref{eq approssimata}-\eqref{ci approssimata}. First of all, from the pointwise convergence \eqref{strong conv} we have $\rho_n\to a\ff+F'(\ff)$ almost everywhere in $\Omega\times (0,T)$. Furthermore, for every $\nu\in\Psi_j$, every $j\leq n$ with $j$ fixed and for every $\chi\in C_0^\infty(0,T)$, we have that
\[
\Tint(\rho_n ,\nu)\chi(t)=\Tint(\tilde\rho_n,\nu)\chi(t).
\]
Passing to the limit in this equation, using \eqref{rhon conv} and \eqref{rho conv}, and on account of the density of $\{\Psi_j\}_{j\geq 1}$ in $H$, we get $\tilde{\rho}(\cdot,\ff)=\rho(\cdot,\ff) = a\ff+F'(\ff)$ recovering \eqref{eq rho}.
Moreover, since $\mn=P_n(\rho_n-J\ast\fn)$, then, for every $\nu\in\Psi_j$, every $k\leq j$ with $j$ fixed and for every $\chi\in C_0^\infty(0,T)$, there holds
\[
\Tint(\mu_n(t),\nu)\chi(t)dt=\Tint(\rho_n-J\ast\fn,\nu)\chi(t)dt.
\]
By passing to the limit in the above identity and using the convergences \eqref{strong conv}, \eqref{mn conv} and \eqref{rho conv}, we eventually get
\[
\mu=a\ff-J\ast\ff+ F'(\ff)=\rho-J\ast\ff.
\]
It still remains to pass to the limit in \eqref{eq approssimata} in order to recover \eqref{eq}. To this aim we multiply \eqref{eq approssimata} by $\chi\in C_0^\infty(0,T)$ and integrate in time between $0$ and $T$.
Using now  \eqref{fn conv},  \eqref{strong conv},  \eqref{rhon conv},  \eqref{fn' conv} and observing that $(\nabla\tilde\rho_n,\nabla\psi)=(\nabla\rho_n,\nabla\psi)$, where we can use \eqref{rhon conv}, we get
\begin{align*}
\nonumber
&\Tint (\ff',\psi)\chi(t)\,dt+\Tint (\nabla\rho(\cdot,\ff),\nabla\psi)\chi(t)\,dt+\Tint (\eps(\ff-\frif),\psi)\chi(t)\,dt
\\
&=\Tint (u\ff,\nabla\psi)\chi(t)\,dt+\Tint (\nabla
J\ast\ff,\nabla\psi)\chi(t)\,dt+ \Tint \langle g,\psi\rangle\chi(t)\,dt.
\end{align*}
This identity holds for every $\psi\in\Psi_j$, where $j$ is fixed, and every $\chi\in C_0^\infty(0,T)$.
Thus a standard density argument entails that
\[
\langle\ff_t,\psi\rangle+(\nabla\rho,\nabla\psi)+(\eps(\ff-\frif),\psi)=(\nabla
J\ast\ff,\nabla \psi)+(u\ff,\nabla\psi)+ \langle g,\psi\rangle,
\]
for every $\psi\in V$ and almost everywhere in $(0,T)$, that is  \eqref{eq}. The initial condition  \eqref{ci} can be easily recovered in a standard way.

Let us now assume that $\ff_0\in H$ such that $F(\ff_0)\in L^1(\Omega)$. Arguing as in~\cite[Proof of Theorem 1]{grass} we can show that there exists a solution also in this case. Of course, it is easy to realize that assumptions (H5) and (H8) suffice. In particular, on account of  \eqref{stima forzante},  \eqref{stima u},  \eqref{forzante h} and  \eqref{stima4}, from \eqref{stima5} we deduce that $F(\ff)\in L^\infty(0,T;L^1(\Omega))$.

\subsubsection{Proof of the energy identity \eqref{energy equality}}
Let us take $\psi=\mu(t)$ in equation \eqref{eq}. This yields
\begin{equation}
\label{ee 1}
\langle \ff_t, \mu\rangle +\eps(\ff-\frif,\mu)+\|\nabla\mu\|^2=(u\ff,\nabla\mu)+\langle g,\mu\rangle.
\end{equation}
By arguing as in~\cite[proof of Corollary 2]{grass} we obtain
\[
\langle \ff_t, \mu\rangle = \langle \ff_t, a\ff+F'(\ff)-J\ast\ff\rangle =\dert \EE
\]
which holds for almost every $t>0$. Then  \eqref{energy equality} follows from \eqref{ee 1}.

\subsubsection{Proof of Proposition \ref{stability}}
\label{unique sect}
Here we generalize~\cite[proof of Proposition 5]{grass trattore}.

Let $\ff_1$ and $\ff_2$ be two solutions to problem \eqref{problema} (cf. Definition \eqref{soluzione debole})
whose initial values are $\ff_{1,0}$ and $\ff_{2,0}$,
respectively. Setting $\ff=\ff_1-\ff_2$ and $\mu_i=a\ff_i+F'(\ff_i)-J\ast\ff_i$, we easily find that $\ff$ solves
\begin{equation}
\label{eq unicita}
\langle\ff_{t},\psi\rangle+(\nabla (\mu_1 - \mu_2),\nabla\psi)+(\eps\ff_,\psi)=(u\ff,\nabla\psi)
\end{equation}
for every $\psi\in V$ and almost everywhere in $(0,T)$, with initial datum $\ff_0=\ff_{1,0}-\ff_{2,0}$.

Choosing $\psi=1$, equation \eqref{eq unicita} yields
\begin{equation}
\label{edomedia-1}
\dert \barf + \eps \barf=0
\end{equation}
with $\barf(0)=\bar{\ff}_0$. Thus we have
\[
|\barf|\leq|\barf_0|.
\]
\noindent
From \eqref{eq unicita} and \eqref{edomedia-1} it follows
\[
\langle\ff_{t}-\barf_t,\psi\rangle+(\nabla(\rho_1-\rho_2),\nabla\psi)+\eps(\ff-\barf,\psi)=(\nabla
J\ast\ff,\nabla \psi)+(u\ff,\nabla\psi),
\]
for every $\psi\in V$ and almost everywhere in $(0,T)$. Choosing $\psi=A^{-1}(\ff-\barf)$, we obtain
\begin{align}
&\frac{1}{2}\dert\|\ff-\barf\|^2_{-1} + \eps\|\ff-\barf\|^2_{-1} +(\rho_1-\rho_2,\ff-\barf)
\nonumber\\
&=(J\ast\ff,\ff-\barf)+(u\ff,\nabla (A^{-1}(\ff-\barf))).\label{uno-due da stimare}
\end{align}

The convective term can be estimated as follows
\begin{equation}
\label{continuita 1}
(u\ff,\nabla(-\Delta)^{-1}(\ff-\barf))\leq\|u\|_{L^\infty}\|\ff\|\|\ff-\barf\|_{-1}\leq c\|u\|_{L^\infty}^2\|\ff-\barf\|_{-1}^2+\frac{c_0}{4}\|\ff\|^2
\end{equation}
while the convolution term can be controlled in this way
\[
(J\ast\ff,\ff-\barf)\leq\frac{c_0}{4}\|\ff\|^2+c\|\ff-\barf\|_{-1}.
\]
Furthermore, on account of (H2), we have
\begin{equation}
\label{continuita 3}
(\rho_1-\rho_2,\ff)=(a\ff+F'(\ff_1)-F'(\ff_2),\ff)\geq c_0\|\ff\|^2.
\end{equation}
On the other hand, recalling Remark~\ref{h4}, we have
\begin{align}
&(\rho_1-\rho_2,\barf) =  (a\ff+F'(\ff_1)-F'(\ff_2),\barf)\nonumber\\
\label{445}
&\leq |\barf| \Bigl[ c+c\Bigl(\mint F(\ff_1)+\mint F(\ff_2)\Bigr) + \|a\|\|\ff\|\Bigr].
\end{align}
Thus we infer the estimate
\[
(\rho_1-\rho_2,\barf)
\leq |\barf| N+ c\barf^2+\frac{c_0}{4}\|\ff\|^2
\leq N|\barf_0|+c\barf^2+\frac{c_0}{4}\|\ff\|^2.
\]

Using \eqref{continuita 1}-\eqref{445}, it follows from \eqref{uno-due da stimare} that
\[
\dert\|\ff-\barf\|^2_{-1}+\frac{c_0}{2}\|\ff\|^2\leq\bigl(N |\barf_0| +c\barf^2\bigr)+c\Vert u\Vert^2_{L^\infty}\|\ff-\barf\|^2_{-1}.
\]
Also, we have
\[
\dert \barf^2=2\barf\dert\barf=-2\sigma\barf^2\leq 0.
\]
Therefore we obtain
\[
\dert\Bigl( \|\ff-\barf\|^2_{-1}+\barf^2 \Bigr)+\frac{c_0}{2}\|\ff\|^2	
\leq c(1+\Vert u\Vert^2_{L^\infty})\bigl( \barf^2+\|\ff-\barf\|^2_{-1} \bigr) + N|\barf_0|
\]
and Gronwall's lemma entails the continuous dependence estimate
\[
\| \ff-\barf\|^2_{-1}+ \barf^2  \leq \bigl( \barf_0^2+\|\ff_0-\barf_0\|^2_{-1} + N|\barf_0| \bigr) e^{cT}
\]
which provides uniqueness.

\section{Nonlocal CHBEG equation: Well-posedness of  \eqref{CHB}}
\setcounter{equation}{0}
\label{EUCHBEG}
 Here we generalize Theorem~\ref{buona posizione} to problem \eqref{CHB}. We use a fixed point theorem. Uniqueness is proven under stronger assumptions on $F$.

\subsection{Assumptions and statements}
In addition to (H1)-(H5) and (H8) we assume the following:
\begin{description}
	\item[(I6)]
		$\la \in L^\infty(\Omega)$ is a non-negative function and $\la^*=\|\la\|_{L^\infty(\Omega)}$.
	\item[(I7)]
		$h\in L^2(\Omega)$.
\end{description}

\begin{definition}
\label{soluzione debole CHB}
Let $\ff_0\in H$ such that $F(\ff_0)\in L^1(\Omega)$ and
$T>0$ be given.
Then $\ff$ is a weak solution to problem \eqref{CHB} on
$[0,T]$ corresponding to $\ff_0$ if:
	\begin{align}
	\label{regolarita CHB 1 def}
	&\ff\in L^\infty(0,T;H)\cap L^2(0,T;V),\;\ff_t\in L^2(0,T;V');\\
	\nonumber
	&\mu=a\ff-J\ast\ff+ F'(\ff) \in \Ltv,
	\end{align}
and
\begin{align}
\label{weakCHBEG}
&\langle\ff_t,\psi\rangle+(\nabla\rho,\nabla\psi)+(\la(\ff-h),\psi)=(\nabla
J\ast\ff,\nabla \psi)+(u\ff,\nabla\psi)+\langle g, \psi\rangle;\\
\nonumber
&\ff(0)=\ff_0,
\end{align}
for every $\psi\in V$ and for almost any $t\in(0,T)$.
Here $\rho$ is given by  \eqref{eq rho}.
	
\end{definition}

\begin{remark}
\label{l2oth}
Observe that, thanks to hypotheses (H8) and (I6)-(I7), if $\phi\break\in\Lth$ then $g-\la\bigl( \phi - h\bigr)\in L^2(0,T;V')$.
\end{remark}

We are now ready to state the existence theorem:

\begin{theorem}
\label{buona posizione CHB}
Let $\ff_0\in H$ such that $F(\ff_0)\in L^1(\Omega)$ and suppose (H0)-(H5), (H8), (I6)-(I7) are satisfied. Then, for every given $ T>0$, there exists a weak solution $\ff$ to problem \eqref{CHB} on
$[0,T]$ corresponding to $\ff_0$. Furthermore, $F(\ff)\in L^\infty(0,T;L^1(\Omega))$ and the following energy equality holds for almost any $t\in (0,T)$:
\begin{equation}
\label{energy equality CHB}
\dert \mathcal{E}(\ff(t)) + \|\nabla\mu\|^2 + (\la(\ff-h),\mu) = (u\ff,\nabla\mu)+\langle g ,\mu\rangle,
\end{equation}
where $\mathcal{E}$ is defined by  \eqref{nlEnergy}.
\end{theorem}

We are able to prove uniqueness of weak solutions under stronger assumptions on $F$, namely,
\begin{description}
	\item[(I8)] There exist $c_6>0$ and $c_7\geq 0$ such that
	\[
	F(s)\geq c_6s^4-c_7,\quad \forall\, s\in\R.
	\]
	Furthermore, we require the existence of a constant $c_8\geq 0$ such that
	\[
	|F'(s)-F'(r)| \leq c_8(1+s^2+r^2)|s-r|,\quad \forall\, s,r\in\R.
	\]
\end{description}

\begin{remark}
It is easy to show that (I8) implies (H3). Moreover, it is straightforward to check that (I8) is satisfied by
the standard double well potential  \eqref{dwp poli} used in~\cite{bertozzi2} (see also~\cite{bertozzi1} and~\cite{miranville pre}).
\end{remark}

We have:
\begin{theorem}
\label{buona posizione CHB2}
Let $\ff_{i,0}\in H$ such that $F(\ff_{i,0})\in L^1(\Omega)$, $i=1,2$ and suppose that (H0)-(H2) and (H4)-(H5), (H8), (I6)-(I8) are satisfied. Then, for any given $T>0$, denoting by $\ff_i$ a weak solution to problem \eqref{CHB} on $[0,T]$ corresponding to $\ff_{i,0}$, there exists a positive constant $N$ such that:
\begin{equation}
\label{stability CHB}
\|\ff_1(t)-\ff_2(t)\|^2_{\#}\leq\|\ff_{1,0}-\ff_{2,0}\|_\#^2\, e^{NT},\quad \forall t\in (0,T).
\end{equation}
\end{theorem}

\begin{remark}
Assumption (I8) implies that $\ff\in L^\infty(0,T;L^4(\Omega))$.
\end{remark}

\subsection{Proof of Theorem~\ref{buona posizione CHB}}
We employ the Schauder fixed point theorem in one of his many variants.
Let us consider the following problem
\begin{equation}
\label{CH forzante}
\begin{cases}
\displaystyle \ff_t -\Delta\mu+\nabla\cdot(u\ff)=G(\phi) &\text{in $\Omega\times(0,T)$}\\[1.4ex]
\displaystyle \mu=a\ff-J\ast\ff+ F'(\ff) &\text{in $\Omega\times(0,T)$}\\[1.4ex]
\displaystyle \frac{\partial\mu}{\partial n}=0, &\text{on $\partial\Omega\times (0,T)$} \\[1.4ex]
\displaystyle \ff(0)=\ff_0, &\text{in $\Omega$}.
\end{cases}
\end{equation}
where $G(\phi)=g(x)-\la(x)(\phi-h)$. For any given $\phi\in L^2(0,T;H)$, Theorem~\ref{buona posizione} entails that there exists a unique weak solution $\ff\in \Ltv\cap L^\infty(0,T;H)$ to \eqref{CH forzante} (just take $\sigma=0$).
Let $X_T= L^2(0,T;H)$, set $\Lambda(\phi):=\ff$ for all $\phi\in X_T$, and denote by $B_R(T)$ the closed ball of $X$ of radius $R$ centered at $0$.

\subsubsection{$\Lambda: B_R(T^*)\to B_R(T^*)$ for some $T^*>0$}
Here we show that there exists $T^*>0$ such that $\Lambda: B_R(T^*)\to B_R(T^*)$. Suppose $\phi\in B_R(T)$. Then the energy identity
 \eqref{energy equality} gives
\begin{equation}
\label{chb s0}
\dert \mathcal{E}(\ff(t)) + \|\nabla\mu\|^2 = (u\ff,\nabla\mu) +  \langle G(\phi) ,\mu \rangle.
\end{equation}
Adding and subtracting $\barm$ to $\mu$ in the last term of the energy equality \eqref{chb s0} yields
\[
(\la(x)(\phi-h),\mu)=(\la(x)(\phi-h),\mu-\barm)+(\la(x)(\phi-h),\barm).
\]
Observe now that
\begin{equation}
\label{chb s1,1}
(\la(x)(\phi-h),\mu-\barm)\leq \la^* \|\mu-\barm\| \|\phi-h\|\leq \la^{*2} C_p^2\bigl( \|\phi\|^2+\|h\|^2\bigr)+\frac{1}{4}\|\nabla\mu\|^2,
\end{equation}
and (cf. \eqref{mn minore di F})
\begin{equation}
\label{chb s1,2}
(\la(x)(\phi-h),\barm)\leq \barm\|\la\|\|\phi-h\|
\leq c \|\la\| (\|\phi\|+\|h\|) \Bigl(\mint F(\ff) + 1 \Bigr).
\end{equation}
Thus \eqref{chb s1,1} and \eqref{chb s1,2} yield
\begin{equation}
\label{chb s2}
(\la(x)(\phi-h),\mu) \leq c + c(\|\phi\|^2 +\|h\|^2)+\frac{1}{4}\|\nabla\mu\|^2+c (\|\phi\|+\|h\|) \mint F(\ff) .
\end{equation}
Similarly, we have
\begin{equation}
\label{chb s3}
\langle g,\mu\rangle  \leq c+\|g\|^2_{V'} +\frac{1}{4}\| \nabla \mu \| ^2+ c\|g\|_{V'} \mint F(\ff).
\end{equation}
By exploiting \eqref{stima u}, on account of \eqref{chb s2} and \eqref{chb s3}, we obtain from \eqref{chb s0} the following
inequality
\begin{equation}
\label{chb s4}
\dert \mathcal{E}(\ff) + \frac{1}{4}\|\nabla\mu\|^2 \leq c+c\|\ff\|^2 + c(\|\phi\|+\|h\|+\|g\|_{V'})
\Bigl( 1+ \mint F(\ff)\Bigr).
\end{equation}
Observe that
\[
\mathcal{E}(\ff_0)
\leq c\|\ff_0\|^2+\mint F(\ff_0)\leq N.
\]
Furthermore, by arguing as in the proof of Theorem~\ref{buona posizione} to obtain \eqref{stima5}, we know that (H3) entails that there is
$\alpha>0$ such that
\[
\EE\geq \alpha\Bigl( \|\ff(t)\|^2+\mint F(\ff(t))\Bigr) -c.
\]
Thus, integrating \eqref{chb s4} with respect to time between $0$ and $t\in(0,T)$, we get (cf.(I7) and (I8))
\begin{align*}
&\alpha\Bigl( \|\ff(t)\|^2+\mint F(\ff(t))\Bigr) + \tint\frac{1}{4}\|\nabla\mu(\tau)\|^2d\tau \\
&\leq
M(1+t)+ c \tint K(\tau) \Bigl( \|\ff(\tau)\|^2 + \mint F(\ff(\tau))\Bigr)d\tau,\quad \forall\,t\in [0,T],
\end{align*}
where $K=\max(1,\|\phi\|+\|h|+\|g\|_{V'})$ is such that
\begin{equation}
\label{boundk}
\Vert K \Vert_{L^2(0,T)} \leq \max\{\sqrt{T}, c\sqrt{T}(1+R)\}.
\end{equation}
Then, an application of the Gronwall lemma provides
\[
\alpha\Bigl( \|\ff(t)\|^2+\mint F(\ff(t))\Bigr) + \tint\frac{1}{4}\|\nabla\mu(\tau)\|^2d\tau \leq
N(T+1)\exp{\Bigl(\tint K(\tau)\Bigr)d\tau} .
\]
Using (H3) once more to control $F(\ff)$, we end up with (cf. \eqref{boundk})
\[
\|\ff(t)\|^2 \leq c+ N(1+T)\exp{\Bigl(\sqrt{T}\| K\|_{L^2(0,T)} \Bigr)}, \quad\forall\,t\in[0,T],
\]
so that $\Lambda: B_R(T^*)\to B_R(T^*)$ for some $T^*>0$.

\subsubsection{$\Lambda$ is continuous and compact}
\label{sezione 2 esistenza chb}
Let $\{\phi_m\}$ be a bounded sequence in $X_{T^*}$ and consider $\ff_m=\Lambda(\phi_m)$. Then, it is easy to prove that every $\phi_m$ satisfies the energy equality \eqref{chb s0} with $G(\phi_m)$ and $\mu_m=a\ff_m+F'(\ff_m)-J\ast\ff_m$. Also, by arguing as in the previous subsection, we find that $\phi_m$
also satisfies \eqref{chb s4} for every $m\in\mathbb{N}$. In particular, we have that
\begin{align}
\nonumber
&\alpha\Bigl( \|\ff_m(t)\|^2+\mint F(\ff_m(t))\Bigr)+ \frac{1}{4}\tint\|\nabla\mu_m(\tau)\|^2d\tau\\
\label{chb t9}
&\leq N(1+T^*)\exp{\Bigl(c\sqrt{T^*} \Vert K\Vert_{L^2(0,T^*)}\Bigr)}=N^*,
\end{align}
where $N$ is independent of $m$.

As a direct consequence of \eqref{chb t9}, we have that
\begin{align}
\label{chb t10}
&\|\ff_m\|_{L^\infty(0,T^*;H)}\leq N^*,\\
\nonumber
&\| F(\ff_m)\|_ { L^\infty(0,T^*;L^1(\Omega))}\leq N^*,\\
\nonumber
&\|\nabla\mu_m\|_{ L^2(0,T^*;V)}\leq N^*
\end{align}
for every $m\in\mathbb{N}$.
Thus, from \eqref{gradfn in V 3}, \eqref{mu in L1} and \eqref{mu in L1 2} we deduce
\begin{align}
\nonumber
&\|\ff_h\|_{L^2(0,T^*;V)}\leq N^*,\\
\label{chb t14}
&\|\mu_h\|_{L^2(0,T^*;V)}\leq N^*.
\end{align}
Then, by comparison in the equation, we find
\begin{equation}
\label{chb t15}
\|\ff_{m,t}\|_{L^2(0,T^*;V')}\leq N^*.
\end{equation}
From \eqref{chb t10}-\eqref{chb t15} we infer the existence of
$\mu\in L^2(0,T^*;V)$ and $\ff\in L^\infty(0,T^*;H)\cap L^2(0,T^*;V)$ with
$\ff_t\in  L^2(0,T^*;V')$,
such that, for a non-relabeled subsequence, we have
\begin{align}
\nonumber
&\ff_m\tow \ff \quad \text{weakly*  in }\; L^\infty(0,T^*;H),\\
\nonumber
&\ff_m\tow \ff \quad \text{weakly  in }\; L^2(0,T^*;V),\\
\nonumber
&\ff_m\to \ff \quad \text{strongly  in } L^2(0,T^*;H)\; \text{and a.e.  in }  \Omega\times(0,T^*),\\
\nonumber
&\mu_m\tow \mu \quad \text{weakly  in }\; L^2(0,T^*;V),\\
\nonumber
&\ff_m\tow \ff_t \quad \text{weakly  in }\quad L^2(0, T^*; V').
\end{align}
In particular, the strong convergence of $\{\ff_m\}$ proves that $\overline{\Lambda(B_R(T^*))}$ is compact in $X_{T^*}$.

In order to prove the continuity of $\Lambda$ we just assume that $\{\phi_m\}$ converges to some $\phi$ in $X_{T^*}$.
Then, on account of the above bounds and of the uniqueness for problem \eqref{CH forzante}, we have that
$\Lambda(\phi_m)$ converges to $\Lambda(\phi)$. We can thus conclude that $\Lambda$ has a fixed point $\ff$ which is
a local weak solution.

\subsubsection{The local solution $\ff$ is global}
\label{3d step section}
We know that $\ff$ satisfies the energy identity \eqref{energy equality CHB} on some maximal interval $(0,T_0)$.

Observe now that (cf.~\eqref{chb s2} and \eqref{chb s3})
\begin{align}
\label{chb u1}
&(u\ff,\nabla\mu)\leq u^{*2}\|\ff\|^2+\frac{1}{4}\|\nabla\mu\|^2 ,\\
\label{chb u3}
&\langle g, \mu\rangle  \leq c+\|g\|^2_{V'} +\frac{1}{4}\| \nabla \mu \| ^2+ c\|g\|_{V'} \mint F(\ff),\\
\label{chb u2}
&(\la h,\mu)\leq c+c\mint F(\ff)+\frac{1}{4}\|\nabla\mu\|^2.
\end{align}
Furthermore, since (H2) implies \eqref{F coerciva}, thanks to Remark~\ref{h4} we get
\begin{equation}
\label{chb u4,1}
F'(\ff)\ff+\frac{a^*}{2}\ff^2\geq F(\ff)-F(0)+\frac{a^*}{2}\ff^2.
\end{equation}
On the other hand, we deduce from \eqref{chb u4,1} that
\begin{equation}
\label{chb u4}
(\la\ff,\mu)\geq(\la\ff,a\ff+F'(\ff)-J\ast\ff)\geq-c-c\|\ff\|^2-\la^*\mint F(\ff).
\end{equation}
Then, integrating \eqref{energy equality CHB} with respect to time from $0$ and $t\in (0,T_0)$ and by exploiting
\eqref{chb u1}, \eqref{chb u3}, \eqref{chb u2} and \eqref{chb u4}, we obtain
{\small\begin{align*}
&\mathcal{E}(\ff(t)) + \frac{1}{4}\tint\|\nabla\mu(\tau)\|^2 d\tau\\
&\leq ct+\mathcal{E}(\ff_0)+\tint\|g(\tau)\|_{V'}^2d\tau
+c\tint\|\ff(\tau)\|^2d\tau
+c\tint(1+\|g(\tau)\|_{V'})\mint F(\ff(\tau)d\tau.
\end{align*}}This gives, on account of (I8), that $T_0=T$. In particular, note that $F(\ff) \in L^\infty(0,T;L^1(\Omega))$.

\subsection{Proof of Theorem~\ref{buona posizione CHB2}}
Consider two solutions, $\ff_1$ and $\ff_2$ to \eqref{CHB} with initial data $\ff_{1,0}$ and $\ff_{2,0}$, respectively.
Set $\ff=\ff_1-\ff_2$, $\mu_i=a\ff_i+F'(\ff_i)-J\ast\ff_i$, $i=1,2$, and observe that
\begin{equation}
\label{eq unicita chb}
\langle\ff_{t},\psi\rangle+(\nabla (\mu_1 - \mu_2),\nabla\psi)+(\lambda\ff,\psi)=(u\ff,\nabla\psi),
\end{equation}
for any $\psi\in V$ and almost everywhere in $(0,T)$, with initial condition $\ff(0)=\ff_{1,0}-\ff_{2,0}$.

Taking $\psi=1$, equation \eqref{eq unicita chb} yields
\[
\dert \barf + \overline{\lambda\ff}=0.
\]
Therefore we have
\[
\langle\ff_{t}-\barf_t,\psi\rangle+(\nabla (\mu_1 - \mu_2),\nabla\psi)+(\lambda\ff-\overline{\lambda\ff},\psi)=(u\ff,\nabla\psi).
\]
Let us take $\psi=A^{-1}(\ff-\barf)$. This gives
\begin{equation}
\label{eq unicita chb3}
\frac{1}{2}\frac{d}{dt} \|\ff-\barf\|_{-1}^2+(\mu_1 - \mu_2,\ff-\barf)+(\lambda\ff-\overline{\lambda\ff},A^{-1}(\ff-\barf))
=(u\ff,\nabla(A^{-1}(\ff-\barf))).
\end{equation}
In order to estimate the reaction term we observe that
\[
(\lambda\ff-\overline{\lambda\ff},A^{-1}(\ff-\barf))
=(\lambda\ff,A^{-1}(\ff-\barf)-\overline{A^{-1}(\ff-\barf)})
\]
so that
\[
(\lambda\ff-\overline{\lambda\ff},A^{-1}(\ff-\barf))
\leq \frac{c_0}{8}\|\ff\|^2 +c\|\ff\|_\#{\color{red}^2} .
\]
\noindent
Furthermore, there holds
\[
(a\ff,\barf) \leq c\barf^2+\frac{c_0}{8}\|\ff\|^2
\]
and, owing to (I8), we have
\begin{align*}
&(F'(\ff_1)-F'(\ff_2), \barf)\leq c_9(\ff(1+\ff_1^2+\ff_2^2), \barf)\\
&\leq \frac{c_0}{8}\|\ff\|^2+c\barf^2(1+\|\ff_1\|^4_{L^4(\Omega)}+\|\ff_2\|^4_{L^4(\Omega)}).
\end{align*}
Assumption (I8) also entails $\ff_i\in L^\infty(0,T;L^4(\Omega))$, $i=1,2$. Thus we obtain
\begin{equation}
\label{unicita chb 103}
(F'(\ff_1)-F'(\ff_2), \barf)\leq N\barf\|\ff\|\leq N\barf^2+\frac{c_0}{8}\|\ff\|^2.
\end{equation}

Collecting \eqref{continuita 1}-\eqref{continuita 3} and \eqref{eq unicita chb3}-\eqref{unicita chb 103}, we get
\begin{equation}
\label{megastima chb unic 1}
\frac{1}{2}\frac{d}{dt} \|\ff-\barf\|_{-1}^2+\frac{c_0}{2}\|\ff\|^2\leq N(\|\ff-\barf\|^2_{-1}+\barf^2).
\end{equation}
Besides we have that
\begin{equation}
\label{megastima chb unic 2}
\dert \barf^2
=2\barf\dert\barf\leq 2\barf\mint\lambda\ff
\leq \frac{c_0}{8}\|\ff\|^2+c\barf^2.
\end{equation}
We add now \eqref{megastima chb unic 2} to \eqref{megastima chb unic 1} and we find
\[
\frac{d}{dt} \bigl(\|\ff-\barf\|_{-1}^2+\barf^2\bigr) \leq N\bigl(\|\ff-\barf\|^2_{-1}+\barf^2\bigr)
\]
so that Gronwall's lemma yields \eqref{stability CHB}.

\section{Nonlocal CHO equation: The global attractor}
\setcounter{equation}{0}
\label{GACHO}

Let us take $g=0$ for the sake of simplicity and suppose that hypotheses (H0)-(H2) and (H4), (H6)-(H7) hold.
Furthermore, replace (H3) and (H5) by, respectively
\begin{description}
\item[(H9)] There exist $c_{9}>0$, $c_{10}>0$ and $q>0$ such that
$$
F''(s)+a(x)\geq c_{9}|s|^{2q}-c_{10},\qquad\forall s\in\R,\text{ a.e. }x\in\Omega;
$$
\end{description}
\begin{description}
\item[(H10)]  $u\in (L^\infty(\Omega) \cap H^1_0(\Omega))^d$.
\end{description}

Then, on account of Theorem \ref{buona posizione} and Proposition \ref{stability},
for any $\ff_0\in L^2(\Omega)$ such that $F(\ff_0)\in L^1(\Omega)$, there exists a unique global weak solution $\ff$.
As a consequence we can define a semigroup on a suitable phase space.

More precisely, we define
\[
\mathcal{Y}:=\bigl\{\ff\in H:\, F(\ff)\in L^1(\Omega)\bigr\},
\]
and we equip it with the distance
\begin{equation}
\label{CHO metrica attrattore}
d(\ff_1,\ff_2)=\|\ff_1-\ff_2\|+\Bigl| \mint F(\ff_1)-\mint F(\ff_2)\Bigr|^{\frac{1}{2}}.
\end{equation}
Then, for any $\ff_0\in \mathcal{Y}$, we set
$$
\ff(t):= S(t)\ff_0,
$$
$\ff$ being the unique global solution to \eqref{problema}.

We will show that the dynamical system $(\mathcal{Y},S(t))$ is dissipative, that is, it has a bounded absorbing set.
Then, following the strategy outlined in \cite{grass trattore}, we prove that the same system has the (connected) global attractor.

\subsection{Bounded absorbing sets}
\label{abs set}
\noindent

\begin{proposition}
\label{S absorbing}
Let (H0)-(H2), (H4), (H6)-(H7), (H9)-(H10) hold. Then $(\mathcal{Y},S(t))$ has a bounded absorbing set.
\end{proposition}

\begin{proof}
We adapt~\cite[proof of Corollary~2]{grass}. By exploiting \eqref{stima forzante}, \eqref{stima u}, \eqref{forzante h} and \eqref{stima4} in the energy equality \eqref{energy equality} (with $g=0$), we obtain
\begin{equation}
\label{stima energia k}
\frac{d}{dt} \mathcal{E}(\ff(t))+\frac{1}{2}\|\nabla\mu\|^2
\leq C_1\|\ff\|^2 + C_2\mint F(\ff)+C_3
\end{equation}
Observe that
\begin{equation}
\label{mu-mumedio1}
(C_2+1)(\mu,\ff-\bar{\ff})
	=(C_2+1)(\mu-\bar{\mu},\ff)
	\leq \frac{1}{4}\|\nabla\mu\|^2+C_p^2(C_2+1)^2\|\ff\|^2.
\end{equation}

On the other hand, taking $\psi=1$ in \eqref{eq}, we get
$$
\frac{d}{dt}\bar{\ff} + \sigma\bar{\ff}=\sigma\bar{\ff}^*.
$$
Therefore, recalling Remark~\ref{growth}, we deduce that for every bounded set $\mathcal{B}$ of $\mathcal{Y}$ there exists $t_0=t_0(\mathcal{B})\geq 0$ and a positive constant $c_{\mathcal{B}}$ such that
\begin{equation}
\label{chb as 98}
F(\bar{\ff}) + \frac{a^*}{2} \|\ff-\bar{\ff}\|^2 \leq a^*\|\ff\|^2+c_{\mathcal{B}},\quad \forall t\leq t_0,
\end{equation}
\begin{equation}
\label{chb as 99}
F(\bar{\ff}) + \frac{a^*}{2} \|\ff-\bar{\ff}\|^2 \leq a^*\|\ff\|^2+ C_4,\quad \forall t> t_0.
\end{equation}

Then, on account of Remark~\ref{perturbation}, using \eqref{chb as 98}-\eqref{chb as 98} we get
\[
\mint F'(\ff)(\ff-\bar{\ff})\geq\mint F(\ff)-\mint F(\bar{\ff}) - \frac{a^*}{2}\|\ff-\bar{\ff}\|^2
\geq \mint F(\ff) - a^*\|\ff\|^2-\mathcal{H}(t),
\]
where $\mathcal{H}(t)=C_{\mathcal{B}}$ for all $t\in [0,t_0]$ and $\mathcal{H}(t)=C_4$ for all $t>t_0$.
Therefore we have
\begin{equation}
\label{mu-mumedio99}
(C_2+1)(\mu,\ff-\barf)
\geq
\frac{1}{2}\EE+\left(C_2+\frac{1}{2}\right) \mint F(\ff)
-C_5(\|\ff\|^2 + \mathcal{H}(t)).
\end{equation}
By putting together \eqref{mu-mumedio1} and \eqref{mu-mumedio99} we obtain
\[
\frac{1}{2}\EE+ C_2\mint F(\ff)+\frac{1}{2}\mint F(\ff)
\leq
\frac{1}{4}\|\nabla\mu\|^2+C_5(\|\ff\|^2 + \mathcal{H}(t)).
\]
Then, on account of (H9), we find
\[
\frac{1}{4}\|\nabla\mu\|^2+C_6(1+\mathcal{H}(t))
\geq
\frac{1}{2}\EE +  C_2\mint F(\ff)+\frac{c_{10}}{2}\mint|\ff|^{2+2q} - C_7\|\ff\|^2
\]
so that
\begin{equation}
\label{stima assorbente 2 fine}
\frac{1}{4}\|\nabla\mu\|^2+C_8(1+\mathcal{H}(t))
\geq C_2\mint F(\ff)+\frac{1}{2}\EE+ C_8 \|\ff\|^2+\frac{c_{10}}{4}\mint|\ff|^{2+2q}
\end{equation}

Finally, we combine \eqref{stima energia k} with \eqref{stima assorbente 2 fine} and obtain
\begin{equation}
\label{dissineqCHO}
\frac{d}{dt}\EE+\frac{1}{2}\EE \leq C_9(1+\mathcal{H}(t)).
\end{equation}
Then the thesis follows from Gronwall's lemma applied for $t\geq t_0$.
\end{proof}

\begin{remark}
\label{regolarita H9}
From \eqref{stima5}, \eqref{dissineqCHO} and (H9) we also deduce that $\ff\in L^\infty(0,+\infty;\break L^{2+2q}(\Omega))$.
\end{remark}

\subsection{Global attractor}

\begin{theorem}
\label{AGCHO}
Let (H0)-(H2), (H4), (H6)-(H7), (H9)-(H10) hold. Then $(\mathcal{Y},S(t))$ has a (connected) global attractor.
\end{theorem}

\begin{proof}
Recalling~\cite{grass trattore} and taking Proposition~\ref{S absorbing} into account, we can show that $S(t)$ is an
eventually bounded semiflow in the sense of~\cite{Ball}. Therefore, we just need to prove that it is compact
(see Ball's result~\cite{Ball} reported in~\cite[Theorem~2]{grass trattore}).

Let $\{\ff_j(0)\}_j$ be a bounded sequence in $\mathcal{Y}_m$. Observe that, by definition of $S(t)$, every $\ff_j(t):=S(t)\ff_j(0)$ satisfies the energy inequality \eqref{stima energia k}. If we integrate it with respect to time between $0$ and $t\in(0,T)$ with $T$ generic, and use \eqref{stima5}, thanks to the Gronwall lemma we obtain the bounds \eqref{fn limitata}-\eqref{Ffn limitata} for some $N>0$ independent of $j$. Besides, arguing as in the previous sections and using \eqref{gradfn in V 3},  \eqref{mu in L1},  \eqref{mu in L1 2}, we deduce the further bounds \eqref{fn in V}, \eqref{mn in V} and \eqref{fn' in V'}.
Therefore, we can find $\ff$, $\mu$, $\tilde\rho$, $\rho$ which satisfy \eqref{esistenza 1} and \eqref{esistenza 2}
such that, for non relabeled subsequences, convergences  \eqref{conv 0}-\eqref{mn conv} and \eqref{fn' conv} hold.
We can now follow~\cite[Proof of Theorem~3]{grass trattore} to show that
$$
\tilde{\mathcal{E}}(\ff_j(t))\to\tilde{\mathcal{E}}(\ff(t)),\qquad \text{a.e. }t>0
$$
where
$$
\tilde{\mathcal{E}}(\ff(t)):=\EE-\tint(u\ff,\nabla\mu)+\eps\tint(\ff-\frif,\mu).
$$
Since
$t\mapsto \tilde{\mathcal{E}}(\ff(t))$ and $t\mapsto \tilde{\mathcal{E}}(\ff_j(t))$ are nonincreasing in time and continuous on $[0,+\infty)$ we obtain $\tilde{\mathcal{E}}(\ff_j(t))\to\tilde{\mathcal{E}}(\ff(t)), \forall t>0$ which yields $\EEj\to\EE,$ for all $t>0$ so that
$\ff_j(t)\to\ff(t)$ strongly in $\mathcal{Y},$ for all $t>0$. Thus the dynamical system has a global attractor.

Finally, observe that, for every $t\in[0,1],$ $\psi\in\mathcal{Y}$ we have
\begin{align*}
\bigl[ d(0,t \psi) \bigr]^2&\leq 2t^2\|u\|^2+2\Bigl| \mint F(t \psi)-\mint F(0)\Bigr|\\&\leq 2c + 2t^2 \|\psi\|^2 +2\mint F(t \psi) + 2\Bigl| \mint F(0)\Bigr|
\end{align*}
and, thanks to Remark~\ref{perturbation}, we obtain
\[
\bigl[ d(0,t \psi) \bigr]^2 \leq 2c+2t \bigl( t + \frac{a^*}{2} \bigr) \| \psi \|^2+2t\Bigl| \mint F(u)-\mint F(0)\Bigr| +4\bigl| F(0) \Bigr|.
\]
This implies that $\mathcal{Y}$ is connected. Hence the global attractor is also connected (see~\cite[Corollary~4.3]{Ball}).\end{proof}

\section{Nonlocal CHBEG equation: The global attractor}
\setcounter{equation}{0}
\label{GACHBEG}
Here we show that problem \eqref{CHB} can also be viewed as a dissipative dynamical system
which possesses a connected global attractor.

Let us assume that assumptions (H0)-(H1), (I6)-(I7) and (H10). Then take $g=0$ for the sake of simplicity
and set
\begin{description}
\item[(I9)]  $F(s):=\frac{1}{4}\bigl(s^2-1)^2.$
\item[(I10)]
		$u\in ((L^\infty(\Omega)\cap H^1_0(\Omega))^d$ such that $\nabla\cdot u\in L^\infty(\Omega)$.	
\end{description}
These further restrictions are due to the peculiar difficulty of this equation, that is,
the uniform control of $|\barf(t)|$ (see \cite{miranville pre} for the local case). Indeed,
the time dependent average cannot be easily controlled by $|\barf(0)|$ as in the case of CHO equation.
Clearly (I9) entails (H2)-(H4) and (I8). Such assumptions ensures that for any $\ff_0\in \mathcal{Y}$, where
\[
\mathcal{Y}:=\bigl\{\ff\in H:\, F(\ff)\in L^1(\Omega)\bigr\},
\]
there is a unique global weak solution $\ff$ owing to Theorems~\ref{buona posizione CHB} and~\ref{buona posizione CHB2}.
Thus we can define a semigroup $S(t):\mathcal{Y} \to\mathcal{Y}$ by setting $\ff(t):=S(t)\ff_0$. Here $\mathcal{Y}$ is
equipped with the metric \eqref{CHO metrica attrattore}.

As in the previous section, we will show that the dynamical system $(\mathcal{Y},S(t))$ has a bounded absorbing set.
Then, the same argument used to prove Theorem~\ref{AGCHO} will lead to the existence of the (connected) global attractor.
However, a crucial preliminary step is the uniform (dissipative) bound of $|\barf(t)|$ for all $t>0$.

\subsection{Dissipative bound for $|\barf(t)|$}
In this section we extend the approach devised in~\cite{miranville pre} for the local CHBEG equation to prove the following:

\begin{proposition}
\label{media controllata a dovere}
Let (H0)-(H1), (I6)-(I7), (I9)-(I10) hold. If $\Vert \ff_0\Vert\leq R$, then there exist $c_R$ and $C_{10}$ such that
\begin{equation}
\label{CHB tratt 15}
|\barf(t)|\leq c_R e^{-ct}+ C_{10},\quad \forall t\geq 0.
\end{equation}
\end{proposition}

\begin{proof}
Choosing $\psi=1$ in  \eqref{weakCHBEG} with $g=0$ we find (cf. \eqref{soluzione debole CHB})
\begin{equation}
\label{test 1 chbt}
\dert\barf+\overline{\lambda(\ff-g)}=0.
\end{equation}
Then, let us multiply \eqref{test 1 chbt} by $\psi\in V$, integrate over $\Omega$ and subtract it to \eqref{weakCHBEG}. This gives
\begin{equation}
\label{CHB trattore 2}
\langle\ff_t-\barf_t,\psi\rangle+(\nabla\mu,\nabla\psi)+(\lambda(\ff-h)-\overline{\lambda(\ff-g)},\psi)=(\ff u,\nabla\psi).
\end{equation}
On the other hand, since $a\barf-J\ast\barf=0$, we have
\begin{equation}
\label{CHB tratt 01}
(\mu,\ff-\barf)=(a(\ff-\barf)-J\ast(\ff-\barf)+F'(\ff),\ff-\barf)\geq (F'(\ff),\ff-\barf)-c\|\ff-\barf\|^2.
\end{equation}
Besides, following~\cite{miranville pre}, we get
\[
(F'(\ff),\ff-\barf)=(F'(\ff)-F'(\barf),\ff-\barf)
\geq c_K\mint\bigl( (\ff-\barf)^4+(\ff-\barf)^2\barf^2\bigr)-c\|\ff-\barf\|^2.
\]
Thus \eqref{CHB tratt 01} yields
\begin{equation}
\label{CHB tratt 0}
(\mu,\ff-\barf)\geq C_{11}\mint\bigl( (\ff-\barf)^4+(\ff-\barf)^2\barf^2\bigr)-c\|\ff-\barf\|^2.
\end{equation}

By arguing as in the proof of Corollary~\ref{buona posizione CHB2} we find
\[
(\lambda(\ff-g)-\overline{\lambda(\ff-g)},(-\Delta)^{-1}(\ff-\barf))\leq 2\|\ff-\barf\|^2+c|\barf| \|\ff-\barf\|+c.
\]
Besides, we have
\begin{equation}
\label{CHB tratt 2}
(u\ff,\nabla(-\Delta)^{-1}(\ff-\barf))\leq c\|\ff\|\|\ff- \barf \|\leq c\|\ff-\barf\|^2+c|\barf| \|\ff-\barf\|.
\end{equation}
\noindent
Let us now choose $\psi=(-\Delta)^{-1}(\ff-\barf)$ in \eqref{CHB trattore 2} and exploit \eqref{CHB tratt 0}-\eqref{CHB tratt 2}. We obtain
$$
\frac{1}{2}\dert\|\ff-\barf\|^2_{-1}+C_{12}\mint\bigl( (\ff-\barf)^4+(\ff-\barf)^2\barf^2\bigr)\leq c\|\ff-\barf\|^2+c|\barf|\|\ff-\barf\|+c.
$$
Observe now that
$$
\|\ff-\barf\|_{-1}^2+\|\ff-\barf\|^2+c\|\ff-\barf\|^2+c|\barf|\|\ff-\barf\|\leq \frac{C_{11}}{4}\mint\bigl( (\ff-\barf)^4+(\ff-\barf)^2\barf^2\bigr)+c.
$$
Therefore we infer
\begin{equation}
\label{CHB tratt 3}
\dert\|\ff-\barf\|^2_{-1}+\|\ff-\barf\|_{-1}^2+\|\ff-\barf\|^2+C_{12}\mint\bigl( (\ff-\barf)^4+(\ff-\barf)^2\barf^2\bigr)\leq C_{12}
\end{equation}
and using the Gronwall lemma, we deduce that
\[
\|\ff-\barf\|_{-1}^2\leq \|\ff_0-\barf_0 \|_{-1}^2e^{-t}+C_{12}.
\]

Let $R>0$. Then we can find $t_0(R)>0$ such that for all $\ff_0$ such that $\Vert\ff_0\Vert \leq R$ there holds
$$
S(t)\ff_0\in B_0:=\bigl\{\psi \in V':\, \|\psi-\bar{\psi}\|^2_{-1}\leq 2C_{12} \bigr\}, \quad\forall\,t\geq t_0.
$$
By integrating \eqref{CHB tratt 3} with respect to time between $t\geq t_0$ and $t+1$, we get
\begin{align}
\label{CHB tratt 5.1}
&\int_t^{t+1}\mint\bigl( (\ff-\barf)^4+(\ff-\barf)^2\barf^2\bigr)\leq C_{13}\\
&\label{CHB tratt 5.2}
\int_t^{t+1}\|\ff-\barf\|^2\leq C_{13}.
\end{align}

\noindent
Choose now $\psi=\ff-\barf$ in \eqref{CHB trattore 2}. This gives
\begin{equation}
\label{CHB trattore 4}
\frac{1}{2}\dert\|\ff-\barf\|^2+(\nabla\mu,\nabla(\ff-\barf))+(\lambda(\ff-h)-\overline{\lambda(\ff-h)},\ff-\barf)=(\ff u,\nabla(\ff-\barf)).
\end{equation}
Since $a\barf-J\ast\barf=0$, thanks to (H2) and Young's lemma, we have
\begin{align}
\nonumber
&(\nabla\mu,\nabla(\ff-\barf))=(F''(\ff)\nabla\ff+a\nabla\ff+(\ff-\barf)\nabla a -\nabla J\ast (\ff-\barf), \nabla\ff)
\\
\label{CHB tratt 6}
&\geq c_0\|\nabla\ff\|^2-c\|\nabla\ff\|\|\ff-\barf\|
\geq \frac{3 c_0}{4}\|\nabla\ff\|^2-c\|\ff-\barf\|^2.
\end{align}

\noindent
Besides, there holds
\[
(\lambda(\ff-h)-\overline{\lambda(\ff-h)},\ff-\barf) \leq c+\lambda^*(\|\ff-\barf\|^2+\|\ff-\barf\| |\barf|).
\]

\noindent
Observing now that
$$
(\ff u,\nabla(\ff-\barf))=-(\nabla\cdot(\ff u),\ff-\barf)=-(u\cdot\nabla\ff ,\ff-\barf)-(\ff\nabla\cdot u ,\ff-\barf),
$$
and recalling (I10), we obtain
\begin{equation}
\label{CHB tratt 8}
(\ff u,\nabla(\ff-\barf))\leq c\|\ff-\barf\|^2+ \frac{c_0}{4}\|\nabla\ff\|^2 + c|\barf| \|\ff-\barf\|.
\end{equation}

\noindent
Taking \eqref{CHB tratt 6}-\eqref{CHB tratt 8} into account, from \eqref{CHB trattore 4} we deduce
\begin{equation}
\label{CHB trattore 5}
\dert\|\ff-\barf\|^2+c_0\|\nabla\ff\|^2  \leq c\|\ff-\barf\|^2 + c|\barf| \|\ff-\barf\|
\leq c+ c\mint\bigl( (\ff-\barf)^4+(\ff-\barf)^2\barf^2\bigr).
\end{equation}
Then, by means of the uniform Gronwall lemma, on account of \eqref{CHB tratt 5.1} and \eqref{CHB tratt 5.2}, we find
\begin{equation}
\label{CHB tratt 9}
\|(\ff-\barf)(t)\|^2\leq C_{14},\qquad\forall t\geq t_0+1.
\end{equation}
Furthermore, by integrating inequality \eqref{CHB trattore 5} with respect to time between $0$ and $t\leq t_0+1$,
we get
$$
\|(\ff-\barf)(t)\|^2\leq C_{15}\left(1 + \tint \mint\bigl( (\ff-\barf)^4+(\ff-\barf)^2\barf^2\right)
$$
and thanks to \eqref{CHB tratt 3} we are led to
\begin{equation}
\label{CHB tratt 10}
\|(\ff-\barf)(t)\|^2\leq c_R,\qquad\forall t\leq t_0+1.
\end{equation}

We are now ready to recover an estimate for $\barf$. Let us rewrite equation \eqref{test 1 chbt} as
\[
\dert\barf+\bar{\lambda}\barf=-\overline{\lambda(\ff-\barf-h)},
\]
so that
$$
\barf(t)=e^{-\bar{\lambda}t}\barf_0-e^{-\bar{\lambda}t}\tint e^{\bar{\lambda}s}\mint\lambda(\ff-\barf-h)\,ds,
$$
and
\begin{equation}
\label{CHB tratt 13}
|\barf(t)|\leq e^{-\bar{\lambda}t}|\barf_0| + e^{-\bar{\lambda}t}\tint e^{\bar{\lambda}s}\mint |\lambda(\ff-\barf-h)|\,ds.
\end{equation}

\noindent
On the other hand, we have
$$
\tint e^{\bar{\lambda}s}\mint |\lambda(\ff-\barf-h)|\,ds \leq \lambda^*|\Omega|^{\frac{1}{2}}\tint e^{\bar{\lambda}s}(\|\ff-\barf\|)\,ds+
\frac{\lambda^*|\Omega|^{\frac{1}{2}}}{\bar{\lambda}}e^{\bar{\lambda}t}\|h\|.
$$
Thus from \eqref{CHB tratt 13} we infer
\begin{equation}
\label{CHB tratt 14}
|\barf(t)|\leq e^{-\bar{\lambda}t}|\barf_0| + C_{16} + e^{-\bar{\lambda}t}\lambda^*|\Omega|^{\frac{1}{2}} \tint e^{\bar{\lambda}s} \|\ff-\barf\|\,ds.
\end{equation}
Finally, for $t\geq t_0+1$, we can use \eqref{CHB tratt 9} and \eqref{CHB tratt 10} to control the last term of the above inequality
\begin{align}
\nonumber
&e^{-\bar{\lambda}t}\lambda^*|\Omega|^{\frac{1}{2}} \tint e^{\bar{\lambda}s} \|\ff-\barf\|\,ds
=e^{-\bar{\lambda}t}\lambda^*|\Omega|^{\frac{1}{2}} \int_0^{t_0+1} e^{\bar{\lambda}s} \|\ff-\barf\|\,ds
\\
\label{CHB tratt 16}
&+e^{-\bar{\lambda}t}\lambda^*|\Omega|^{\frac{1}{2}} \int^t_{t_0+1} e^{\bar{\lambda}s} \|\ff-\barf\|\,ds
\leq \lambda^*|\Omega|^{\frac{1}{2}}c_R e^{-\bar{\lambda}(t-(t_0+1))}+ C_{17}.
\end{align}
Therefore, from \eqref{CHB tratt 14} and \eqref{CHB tratt 16} we deduce  \eqref{CHB tratt 15}.
\end{proof}

\subsection{Bounded absorbing sets and the global attractor}

Thanks to Proposition~\ref{media controllata a dovere} and arguing as in the proof of Proposition~\ref{S absorbing},
we can prove the following
\begin{proposition}
\label{S absorbing CHB}
Let (H0)-(H1), (I6)-(I7), (I9)-(I10) hold. Then $(\mathcal{Y},S(t))$ has a bounded absorbing set.
\end{proposition}

Finally, adapting the proof of Theorem~\ref{AGCHO}, we get
\begin{theorem}
Let the assumptions of Proposition~\ref{S absorbing CHB} hold. Then $(\mathcal{Y},S(t))$ has the (connected) global attractor.
\end{theorem}


\section*{Acknowledgments} This work was the subject of the first author's Master's thesis at the Politecnico di Milano. His work was partially supported by the Engineering and Physical Sciences Research Council [EP/L015811/1]. The second author is a member of the Gruppo Nazionale per l'Analisi Matematica, la Probabilit\`{a} e le loro Applicazioni (GNAMPA) of the Istituto Nazionale di Alta Matematica (INdAM).


\medskip
\medskip

\end{document}